\documentclass[a4paper,11pt]{amsart}
\usepackage{amsmath,amssymb,amsthm,amsfonts}
\usepackage{color}
\usepackage[all]{xy}

 \usepackage{hyperref}

\newcommand{\kk}{\mathbb{K}}

\newcommand{\CC}{\mathbb{C}}
 
\newcommand{\OO}{\mathcal{O}}

\newcommand{\deltabar}{\bar{\partial}}

\DeclareMathOperator{\im}{Im}
\DeclareMathOperator{\Hom}{Hom}
\DeclareMathOperator{\Art}{Art}
\DeclareMathOperator{\Set}{Set}
\DeclareMathOperator{\Spec}{Spec}
\DeclareMathOperator{\Id}{Id}
\DeclareMathOperator{\rk}{rk}
\DeclareMathOperator{\MC}{MC}
\DeclareMathOperator{\Def}{Def}

\DeclareMathOperator{\End}{End}

\DeclareMathOperator{\Pic}{Pic}

\newtheorem{theorem}{Theorem}[section]
\newtheorem{corollary}[theorem]{Corollary}

\newtheorem{proposition}[theorem]{Proposition}
\newtheorem{lemma}[theorem]{Lemma}
\theoremstyle{definition}
\newtheorem{definition}[theorem]{Definition}
\newtheorem{remark}[theorem]{Remark}

\definecolor{rosso}{RGB}{162,0,0}
\definecolor{verde}{RGB}{0,100,0}
\definecolor{blu}{RGB}{0,0,162}

\title[On the local structure of the Brill-Noether locus]{On the local structure  of the Brill-Noether locus of locally free sheaves on a smooth variety}
\author{Donatella Iacono}
\address{\newline Dipartimento di Matematica,
\newline  Universit\`a degli Studi di Bari Aldo Moro,
\hfill\newline Via E. Orabona 4,
70125 Bari, Italy.}
\email{donatella.iacono@uniba.it}

\author{Elena Martinengo}
\address{
\newline
Dipartimento di Matematica ``Giuseppe Peano'',
\hfill\newline
 Universit\`a degli Studi di Torino,
 \hfill\newline
 via Carlo Alberto 10,
10123 Torino, Italy}
\email{elena.martinengo@unito.it}


\begin{document}
\begin{abstract}

 We study  the functor $\Def_E^k$ of infinitesimal deformations of a locally free sheaf $E$ of $\OO_X$-modules  on a smooth variety $X$, such that at least $k$ independent sections lift to the deformed sheaf.
We deduce some information on the $k$-th Brill-Noether locus of $E$, such as the description of the tangent cone at some singular points, of the tangent space at some smooth ones and some links between the smoothness of the functor $\Def_E^k$ and the smoothness of some well know deformations functors and their associated moduli spaces. 
As a tool for the investigation of $\Def_E^k$, we study infinitesimal deformations of the pairs $(E,U)$, where $U$ is a linear subspace of sections of $E$. We generalise many classical
results concerning the moduli space of coherent systems to the case where E has any rank and X any dimension.
This includes a description of its tangent space and the link between smoothness and the injectivity of the Petri
map.

\end{abstract}

\maketitle 
\tableofcontents
\addtocontents{toc}{\protect\setcounter{tocdepth}{1}}

\section*{Introduction}

Let $X$ be a smooth projective  variety of dimension $n$ and  $E$ be a locally free sheaf of $\OO_X$-modules on $X$. We are interested in the deformations of $E$ together with some of its sections.

Much is known about the classical case of a line bundle $L$ on a smooth projective curve $C$.
Classical Brill-Noether theory concerns with is concerned with the subvarieties $W_d^k(C)$  of $\Pic^d(C)$ of linear systems on $C$ of degree $d$ and projective dimension at least $k$ or equivalently of line bundles on $C$ of degree $d$ having at least $k+1$ independent  sections. Properties of $W_d^k(C)$ like non emptiness, connectedness, irreducibility and dimension were largely investigated, successfully  determined and summarised in \cite{ACGH}.

During the last few years, several generalisations of this problem were investigated.
Many efforts have been carried out to analyse the moduli space of stable vector bundles of fixed rank and degree on a curve, having at least $k$ independent sections.
In this context, less is known, about dimension, connected components and singular locus, see for instance \cite{Gt} and reference therein for an overview on this case.

To overcome the difficulties of the  Brill-Noether theory of vector bundles, the concept of coherent systems was introduced,  see \cite{nt,king,lepoit, bgmm03, bgmm, newsteadquestion} and reference therein. 
A coherent system on an algebraic variety is the pair of a vector
bundle $E$ of fixed rank $n$ and degree $d$ together with a linear subspace $U$ of sections of $E$ of dimension $k$.
There is a notion of stability which allows to construct the moduli spaces of coherent systems fixing the parameters $(n,d,k)$, see \cite{king,lepoit}.
The relation between these moduli spaces and Brill-Noether theory is obvious: 
any vector bundle which occurs as part of a coherent system must have at least a prescribed number of  independent sections. Conversely, a vector bundle with a prescribed number of linearly independent sections
determines, in a natural way, a coherent system.
This defines a forgetful map $(E,U) \to E$ surjecting to the corresponding Brill-Noether locus. This map is an obvious generalisation of the classical projection $G^k_d(C) \to W^k_d(C)$, where $G^k_d(C)$ is the variety that parametrises the linear systems of degree $d$ and projective dimension exactly $k$ on a curve $C$. In many of the above works [loc.cit.],   the aim is to deduce as much information as possible on the Brill-Noether loci  from the moduli spaces of coherent systems that are easier to describe. 

In \cite{he}, the infinitesimal study of the moduli space of coherent system on varieties was carried out, allowing a description of the tangent space and an obstruction space, see also \cite[Section 3]{bgmm}.

On the other hand, in \cite{pardini} the authors generalised
Brill-Noether theory to line bundles on smooth projective varieties of dimension greater than one. They  were able to prove non emptiness and find the dimension of the Brill-Noether loci of a curve $C$ over a smooth surface $X$ of maximal Albanese dimension, under the hypothesis that a properly defined Brill-Noether number is positive and under some mild additional assumptions.

Finally, the general case of vector bundles  on varieties of higher dimension is still quite mysterious.
In \cite{costa}, the authors prove the existence of a Brill-Noether type stratification of the moduli spaces of stable vector bundles on a smooth projective varieties with fixed Chern classes under the assumption that all the cohomology groups of degree greater than one vanish. 
 Some properties known in the classical Brill-Noether theory for line bundles on curves are expected to hold for the  general case of vector bundles on higher dimensional varieties too. Many of them are known to experts but often there is no reference for them.
 
\medskip

 In this paper, we are interested in the  infinitesimal deformations of a locally free sheaf $E$ of $\OO_X$-modules on a smooth variety $X$ that preserve at least a prescribed number of independent sections. As a tool for this analysis we also study infinitesimal deformations of $E$ with a fixed subspace of global sections $U$. We focus on the local study of the Brill-Noether loci and of the moduli space  of  local system respectively,  on all   properties we can predict using deformation theory and we do not concentrate on its global structure.

It is nowadays almost accepted that the most appropriate way to analyse
locally a moduli problem through infinitesimal deformations is via derived algebraic geometry.
 The philosophy behind, called the Deligne's principle, may be formulated in the following way: every deformation problem on a field of characteristic zero is controlled  by a differential graded  Lie  algebra (dgLa) via the Maurer-Cartan equation and gauge equivalence. 
A rigorous proof of  this philosophy was  independently  given by Lurie \cite{lurie}  and Pridham \cite{pridham} via an equivalence of infinity categories between dgLa and formal moduli problem. 
 The dgLa associated with a certain deformation problem is defined only up to quasi-isomorphism and it encodes much information about the problem. For instance, its first cohomology group coincides with the Zariski tangent space of the  moduli space   and its second cohomology group  is an obstruction space for the problem.

This approach has been successfully applied in many cases such as deformations of  locally free sheaves \cite{Fuk}, locally free sheaves with prescribed cohomological dimensions \cite{ElenaFormal}, coherent sheaves \cite{FIM}, pairs of manifold and coherent sheaves  \cite{Dona-Marco}.

Inspired by this philosophy and following \cite{elenatesi}, we find the dgLa controlling infinitesimal deformations  of a pair $(E,U)$ as above and are able to recover and generalise some classical results. 
Later we deduce from this study and the obvious forgetful maps of functors of deformations, some information about infinitesimal deformations of $E$ such that at least a certain number of independent sections lift.
Note that such deformations are basically more difficult to study than the deformations of $(E,U)$,  since they do not define a deformation functor and they do not classically fall within the reach of Deligne's principle.  The dgLas's approach  completely describes deformations of the pairs $(E,U)$, for any rank of $E$ and any dimension of $X$, and provides some information on deformations of $E$ with at least a certain number of independent sections.  An approach based on dgl pairs was used in \cite{Budur-Wang} and \cite{Budur-Rubio} to analyse the cohomology jump functors.  There, the authors extend Deligne’s principle to deformations with cohomology constraints and in \cite{doan} this extension is fully proved.  The fact that these are not deformation functors is still an obstacle to a full understanding of the problem.

Our main motivation here was to test the power of the derived techniques in this very classical context, where the classical theory has not yet answered all questions and predict some properties the associated moduli spaces has to satisfy.
In particular, using this  alternative approach, we are able to show some results, probably expected by the experts.  This can be considered a first step to tackle this kind of deformations.

Here, we restrict our attention  to holomorphic locally free sheaf $E$ of $\OO_X$-modules  on a smooth complex manifold.
 In \cite{IM}, we extend these techniques to describe a dgLa that controls deformations of $(E,U)$ on any algebraically closed field of characteristic zero $\kk$, 
 using Thom-Whitney constructions.
This would offer the possibility to broaden the classical results of Brill-Noether theory over any $\kk$.  Once we have an explicit description of the dgLa, we aim to apply this powerful approach to investigate formality or abelian homotopy  and deduce information about smoothness   of the Brill-Noether locus.

\medskip 

The article is organised as follows. For the convenience of the reader, we first collect some basic notions on deformation functors, differential graded Lie algebras and the link between them. 

In the second section, we briefly recall the definition of deformations of a locally free sheaf and exhibit the dgLa that controls these deformations following \cite{Fuk}.

The third section is finally devoted to the study of deformations of pairs $(E,U)$, where $E$ is a locally free sheaf of $\OO_X$-modules on a smooth variety $X$ and $U$ is a linear subspace of its sections. For the basic definitions and the identification of the dgLa that controls these deformations, we follow \cite{elenatesi}. 
Moreover, we are able to generalise some classical results: the condition for a section of a locally free sheaf to be extended to a first order deformation and the description of the image of the Petri map (Proposition \ref{prop.sezione si estende se cup product va a zero}).  We also describe the tangent space to the functor of deformations of $(E,U)$ in the case $U=H^0(E)$ (Corollary \ref{Cor. TDef(E,H0(E))}).

In Section $4$, we specialize the study of deformations of pairs $(E,U)$ to the case of a smooth curve. In Lemma \ref{lemma.equiv h2 zero e petri iniettiva}, we find two equivalent conditions to the injectivity of the Petri map, which is known to be crucial in the classical study of Brill-Noether loci. 
In Proposition \ref{prop.beta e liscezza}, we compute the dimension of the tangent space to the functor of deformations of a pair $(E,U)$ and find equivalent conditions to its smoothness, generalising the classical results concerning the variety $G^r_d(C)$.

Section $5$ is devoted to our main aim. Let $E$ be a locally free sheaf $E$ of $\OO_X$-modules on a smooth variety $X$, such that $\dim H^0(X,E) \geq k$. We study infinitesimal deformations of $E$ such that at least $k$ independent sections of $E$ lift. 
First we define this kind of deformations and observe that the functor $\Def_E^k$ associated to them is not a deformation functor  in the sense of Definition \ref{def. def functor}. Theorem \ref{teorema tangente Def^r} describes the first order deformations $\Def_E^k(\CC[\epsilon])$ and the vector space generated by them, suggesting that the locally free sheaves with at least $k+1$ independent sections are singular points in the moduli space of  sheaves with at least $k$ sections. 
Propositions \ref{prop.alpha sur equivalenza di liscezze} and \ref{prop.H2=0 liscezza} deal with the smoothness of the functor $\Def_E^k$, linking it with other known conditions.   
\medskip

 We indicate with $\kk$ an algebraic closed field of characteristic zero.  
We will work often over the field $\CC$ of complex numbers.  We denote by $\kk[\epsilon]$ the ring of dual number, meaning $\epsilon^2=0$.

\subsection*{Acknowledgements}  
 We wish to thank Marco Manetti for some inspiring mathematical discussions,  Andrea Petracci for answering our questions  and Peter Newstead for precious comments and suggestions on a first draft of this paper. 
We also thank Nero Budur for having drawn our attention to \cite{Budur-Wang} and \cite{Budur-Rubio}.
We also warmly thank the referee for all the useful comments and suggestions which have improved the paper.

\section{Preliminaries on deformation functors} 
 The first part of this section is dedicated to some preliminaries on functors of Artin rings and deformation functors we will need in the article. 
 In the second part, we introduce
the basic definitions of differential graded Lie algebras and the deformation theory associated to them. 
For a complete presentation of the topics, we refer the reader to  \cite{Schlessinger, FantechiManetti, Man Pisa, Man Roma, ManettiSeattle}.
\medskip

\bigskip 
\subsection{Theory of deformation functors}

\begin{definition} 
A \emph{functor of Artin rings} is a covariant functor $F :\Art_\kk \to \Set$, 
such that $F(\kk)=* $, where $*$ is the one point set, 
$\Set$ denotes the category of sets in a fixed universe and $\Art_{\kk}$ the category of local Artinian $\kk$-algebras with residue fields $\kk$.
\end{definition}
Consider the following diagram whose objects and arrows are in $\Art_\kk$
\[ \xymatrix{ B \times_A C \ar[r]\ar[d]  & C \ar[d]\\ B \ar[r] & A,}  \]
applying a functor $F: \Art_\kk \to \Set$, we get a map
\[\eta: F(B\times_A C) \to F(B) \times_{F(A)} F(C).  \]

\begin{definition} \label{def. def functor}
A functor of Artin rings $F$ is called a \emph{deformation functor} if it satisfies the following conditions:
\begin{itemize}
\item $\eta$ is surjective whenever $B\to A$ is surjective
\item $\eta$ is an isomorphism whenever $A=\kk$. 
\end{itemize}
The functor $F$ is called \emph{homogeneous}, if $\eta$ is an isomorphism, whenever $B\to A$ is surjective.
\end{definition}
The name comes from the fact that most functors arising by deforming geometric objects are deformation functors and some of them are actually homogeneous.
\begin{remark}
Our definition of a deformation functor follows \cite{Man Pisa}. The first condition here is the classical Schlessinger's condition (H1) in \cite{Schlessinger}, while the second is slightly more restrictive than the (H2) of \cite{Schlessinger}. We assume these conditions
because they guarantee good properties for tangent spaces and obstruction theory as stated in Proposition \ref{prop.sp tg} and Theorems \ref{teo.univ ob} and \ref{teo.univ rel ob}.
For more details see \cite[Example 6.8]{FantechiManetti}.
\end{remark}

\begin{proposition}  \label{prop.sp tg}
Let $F$ be a deformation functor, the set $t_F=F(\kk[\epsilon]) $ has a natural structure of $\kk$-vector space. If $\varphi: F \to G$ is a morphism of deformation functors the induced map $\varphi: t_F \to t_G$ is linear. 
\end{proposition}

\begin{definition}
Let $F$ be a deformation functor.
The vector space $t_F=F(\kk[\epsilon])$ is called the \emph{tangent space} to $F$. 
\end{definition}

\begin{definition} \label{def.smoothness}
A morphism of functors of Artin rings $\varphi: F \to G$ is called \emph{smooth} if 
for every surjection $B \to A$ in $\Art_\kk$, the map
\[F(B) \to G(B) \times_{G(A)} F(A)\] is also surjective.  

A functor of Artin rings $F$ is \emph{smooth} if the morphism $F \to *$ is smooth, i.e. if $F(B)\to F(A)$ is surjective for every surjective morphism $B\to A$ in $\Art_\kk$. 
\end{definition}

\begin{remark}
Note that, if $\varphi: F \to G$ is smooth, then the induced map $F(A) \to G(A)$ is surjective for all $A\in \Art_\kk$.
\end{remark}

\begin{proposition}\label{prop liscezza funtori}
Let $\varphi: F \to G$ a smooth morphism of functors of Artin rings. Then: 
$F$ is smooth if and only if $G$ is smooth. 
\end{proposition}
\begin{proof}

Let $ B \to A$ be a surjection in $\Art_\kk$.

$(\Rightarrow):$
Consider the following commutative diagram
\[  \xymatrix{  F(B)  \ar@{>>}[r] \ar@{>>}[d]^{\varphi_B} & F(A) \ar@{>>}[d]^{\varphi_A}  \\
G(B) \ar[r]    & G(A) } \]
in which the vertical arrow are surjective by smoothness of $\varphi$ and the upper horizontal arrow is surjective by smoothness of $F$. Thus the lower arrow is surjective too and $G$ is smooth. 

$(\Leftarrow):$ 
Let $f_A \in F(A)$. Since $G$ is smooth, there exists an element $g_B \in G(B)$ that lifts $\varphi_A(f_A) \in G(A)$. By smoothness of $\varphi$, the element $(f_A, g_B) \in F(A) \times_{G(A)} G(B)$ has a pre-image $f_B\in F(B)$, that assures the surjectivity of $F(B) \to F(A)$ and then the smoothness of $F$. 

\end{proof}

We now introduce the notion of obstruction theory that is crucial in the study of deformations.

By a \emph{small extension} in $\Art_\kk$ we mean an exact sequence
\[ e: 0\to J \to B \stackrel{\varphi}{\to} A \to 0\]
where $\varphi: B \to A$ is a morphism in $\Art_\kk$ and $J$ is an ideal of $B$ 
 annihilated by the maximal ideal $m_B$. In particular $J$ is a finite dimensional vector space over $B/m_B = \kk$.
\begin{definition}
Let $F$ be a functor of Artin rings. An \emph{obstruction theory} $(V, v_e)$ for $F$ is the
data of a $k$-vector space $V,$ called \emph{obstruction space}, and for every small extension in $\Art_\kk$:
\[  e: 0\to J \to B \stackrel{\varphi}{\to} A \to 0\]
of an obstruction map $v_e : F(A) \to V \otimes_\kk J$ satisfying the following properties:
\begin{itemize}
\item If $ a\in F(A)$ can be lifted to $F(B)$ then $v_e(a) = 0$.
\item (base change) For every morphism $\alpha: e_1\to e_2$ of small extensions, i.e. for every commutative diagram
\begin{equation} \label{morph of small ex}\xymatrix{  0  \ar[r]& J_1 \ar[r]\ar[d]^{\alpha_J} & B_1 \ar[r]\ar[d]^{\alpha_B} &  A_1 \ar[r]\ar[d]^{\alpha_A} & 0 \\
0  \ar[r]& J_2 \ar[r] & B_2 \ar[r] &  A_2 \ar[r] & 0  } \end{equation}
we have $v_{e_2}(\alpha_A(a))=(Id_V\otimes \alpha_J)(v_{e_1}(a))$ for every $a\in F(A_1)$.
\end{itemize}
An obstruction theory $(V, v_e)$ for $F$ is called \emph{complete} if the converse of
the first condition holds, i.e. the lifting exists if and only if the obstruction vanishes.
\end{definition}
\begin{remark}
Note that, if $F$ is smooth then all the obstruction maps are trivial.
The inverse holds if the obstruction theory is complete. 

Clearly if $F$ admits a complete obstruction theory then it admits infinitely ones; it is in
fact sufficient to embed $V$ in a bigger vector space. One of the main interest is to look for
the smallest complete obstruction theory.
\end{remark}

\begin{definition} A \emph{morphism of obstruction} theories $(V, v_e) \to (W,w_e)$ is a linear map
$\varphi: V \to W$, such that $w_e = \varphi(v_e)$ for every small extension $e$. 
An obstruction theory $(O_F , ob_e)$  for $F$ is called \emph{universal} if for every obstruction theory $(V, v_e)$ for $F$ there exists a unique morphism $(O_F , ob_e)\to(V, v_e)$.
\end{definition}

\begin{theorem}\cite[Theorem 3.2]{FantechiManetti} \label{teo.univ ob}
Let $F$ be a deformation functor, then there exists a universal obstruction
theory $(O_F , ob_e)$ for $F$. Moreover the universal obstruction theory is complete and every element of the vector space $O_F$ is of the form $ob_e(a)$ for some small extension
\[ e: 0\to \kk \to B \to A \to 0\] 
and some $a \in F(A)$.
\end{theorem}

In the following, we will also need the notion of relative obstruction theory. 

\begin{definition}
Let $\varphi: F\to G$ be a morphism of functors of Artin rings and suppose
$G$ to be a deformation functor. A \emph{relative obstruction theory} $(V,v_e)$ for $\varphi$ is the data of 
\begin{itemize} 
\item $\kk$-vector space $V,$ called \emph{obstruction space},
\item for every small extension in $\Art_\kk$:
\[  e: 0\to J \to B \stackrel{\varphi}{\to} A \to 0\]
of an obstruction map $v_e : F(A) \times_{G(A)} G(B) \to V \otimes_\kk J$ satisfying the following properties:
\begin{enumerate}
\item If $(a, \beta) \in  F(A) \times_{G(A)} G(B)$ can be lifted to $F(B)$ then $v_e( a, \beta) =0$. 
\item (base change) For every morphism $\alpha: e_1\to e_2$ of small extensions, i.e. for every commutative diagram as (\ref{morph of small ex}), the following diagram is also commutative
\[ \xymatrix{  F(A_1) \times_{G(A_1)} G(B_1)  \ar[r]^-{v_{e_1}} \ar[d]^{(\alpha_A,\alpha_B)} & V\times_{\kk} J_1  \ar[d]^{Id_V \otimes \alpha_J}  \\
F(A_2) \times_{G(A_2)} G(B_2)  \ar[r]^-{v_{e_2}}  &  V\times_{\kk} J_2. } \]

A relative obstruction theory is called \emph{complete}  if the converse of
the first condition holds, i.e. the lifting exists if and only if the obstruction vanishes.
\end{enumerate}
\end{itemize}

\end{definition}

\begin{remark}
Note that, if $\varphi: F\to G$ is smooth then all the relative obstruction maps are trivial. The inverse holds if the relative obstruction theory is complete.
\end{remark}
 
\begin{theorem}\cite[Theorem 3.2]{FantechiManetti} \label{teo.univ rel ob}
Let $\varphi: F \to G$ be a morphism of deformation functors, then there exists a unique universal relative obstruction theory for $\varphi$. 
\end{theorem}

\begin{theorem}\cite[Proposition 2.17]{Man Pisa} \label{teo. liscio se tangete suriettivo e iniettivo ostruzione}
Let $\varphi: F \to G$ be a morphism
of deformation functors and $\varphi': (V, v_e) \to (W, we) $  be a compatible morphism between obstruction theories. If $(V, v_e)$ is complete, $\varphi': V \to W$ is injective and $t_\varphi: t_F\to t_G$ is surjective then $\varphi$ is smooth.
\end{theorem}

 \bigskip 
\subsection{Differential graded Lie algebras and deformation functors}

\begin{definition}
A \emph{differential graded Lie algebra}, briefly a \emph{dgLa}, is the data $(L,d,[\ ,\ ])$, where $L=\bigoplus_{i\in \mathbb {Z}} L^i$ is a $\mathbb  Z$-graded vector space over $\kk$, $d:L^i \rightarrow L^{i+1}$ is a linear map, such that $d \circ d=0$, and $[\ ,\ ]:L^i \times L^j \rightarrow L^{i+j}$ is a bilinear map, such that:
\begin{enumerate}
\item[-] $[\ ,\ ]$ is graded skewsymmetric, i.e. $[a,b]=-(-1)^{\deg a\deg b}[b,a]$,  
\item[-] $[\ ,\ ]$ verifies the graded Jacoby identity, i.e. $[a,[b,c]]=[[a,b],c]+(-1)^{\deg a\deg b}[b,[a,c]]$,
\item[-] $[\ ,\ ]$ and $d$ verify the graded Leibniz's rule, i.e. $d[a,b]=[da,b]+(-1)^{\deg a}[a,db]$, 
\end{enumerate}
for every $a, b$ and $c$ homogeneous.
\end{definition}

\begin{definition}
Let $(L,d_L, [\ ,\ ]_L)$ and $(M, d_M, [\ ,\ ]_M)$ be two dgLas, a \emph{morphism of dgLas}  $\varphi:L\to M$ is a degree zero linear morphism that commutes with the brackets and the differentials.  

A \emph{quasi-isomorphism} of dgLas is a morphism of dgLas that induces an isomorphism in cohomology. 
\end{definition}

Let $L$ be a differential graded Lie algebra, then there is a deformation functor $\Def_L: \Art_{\kk}\to \Set$ canonically associated to it, as follows
\begin{definition}
For all $(A,\mathfrak{m}_A)\in \Art_{\kk}$, we define:
$$\Def_L(A)=\frac{\MC_L(A)}{\sim_{\textrm{gauge}}},$$
where:
$$\MC_L(A)=\left\{x\in L^1\otimes \mathfrak{m}_A \mid dx+\frac{1}{2}[x,x]=0\right\}$$
and the gauge action is the action of $\exp (L^0\otimes \mathfrak{m}_A)$ on $\MC_L(A)$, given by:
$$e^a * x= x+\sum_{n=0}^{+\infty} \frac{([a,-])^n}{(n+1)!}([a,x]-da).$$
\end{definition} 
We recall that the tangent to the deformation functor $\Def_L$ is the first cohomology space $H^1(L)$ of the dgLa $L$. Moreover, a complete obstruction theory for the functor $\Def_L$ can be naturally defined and its obstruction space is the second cohomology space $H^2(L)$ of the dgLa $L$. 

If the functor of deformations of a geometrical object $\mathcal{X}$ is isomorphic to the deformation functor associated to a dgLa $L$, then we say that $L$ controls the deformations of $\mathcal{X}$. 

By definition, any morphism $\varphi: L \to M$, induces a morphism $\varphi: \Def_L \to \Def_M$, that is an isomorphism whenever $\varphi$ is a quasi-isomorphism.

\section {Deformation of locally free sheaves}

Let $X$ be a smooth projective variety of dimension $n$ and  $E$  a locally free sheaf of $\mathcal{O}_X$-modules on $X$.  First of all we recall some notions about the deformations of the locally free sheaf $E$.

\begin{definition}
Let $A$ be a local Artinian $\kk$-algebra with residue field $\kk$.
An \emph{infinitesimal deformation} of $E$ over $A$ is a locally free sheaf $E_A$ of $\OO_X \otimes A$-modules over $X\times \Spec A$, with a morphism $\pi_A: E_A \to E$ such that the obvious restriction of scalars $\pi_A: E_A \otimes_A \kk \to E$ is an isomorphism. The deformation will be indicated with $(E_A, \pi_A)$ or, shortly, with $E_A$.

Two of such deformations $E_A$ and $E'_A$ are \emph{isomorphic} if there exists an isomorphism $\phi$ of sheaves of $\OO_X \otimes A$-modules that makes the following diagram commutative:
\begin{equation} \label{def(E)}
\xymatrix{ E_A \ar[rr]^{\phi} \ar[dr]_{\pi_A}  & & E'_A  \ar[dl]^{\pi'_A} \\
&E.&  }\end{equation}

The \emph{functor} of infinitesimal deformations of $E$ is
\[ \Def_E: \Art_\kk \to \Set.\] 
\end{definition}

It is classically known that $\Def_E$ is a deformation functor. Moreover its tangent space is $t_{\Def_E}=\Def_E(\kk[\epsilon])= H^1(X,\End (E))$ and the obstructions are contained in $H^2(X, \End (E))$, see for example \cite{Sernesi}.

\medskip

 Let $E$ be a locally free sheaf of $\OO_X$-modules on $X$ and let $\End E$ be the locally free sheaf of its endomorphisms.
Over the ground field $\CC$, consider 
\[ A_X^{0,*}(\End E) := \bigoplus_{i=0}^n A_X^{0,i}(\End E):= \bigoplus_{i=0}^n \Gamma\left(X, {\mathcal A}_X^{0,i}(\End E)\right), \] 
the graded vector space of global sections of the sheaf of differential forms with values on the sheaf $\End E$. 
The Dolbeault differential on forms and the bracket defined as the wedge product on forms and the composition of endomorphisms induce a structure of dgLa on it.  
It is well known that this dgLa is the one that controls the deformations of $E$.
\begin{proposition}\cite [Theorem 1.1.1]{Fuk} \label{dgla che governa Def E}
The dgLa $A_X^{0,*}(\End E)$ controls deformations of $E$. 
The isomorphism of functors is given, for all $A\in \Art_{\CC}$, by 
$$\begin{array}{rrll}
\Psi: & \Def_{A_X^{0,*}(\End E)}(A)& \longrightarrow &\Def_{E}(A) \\
&x& \longrightarrow & \ker(\deltabar+x)
\end{array}$$
\end{proposition}
In particular,  the tangent space to $\Def_{E}$ is $\Def_{E}(\CC[\epsilon])=H^1(X, \End E)$ and the obstructions to deformations are contained in $H^2(X, \End E)$,  that fits in the classical picture.

\begin{remark} \label{rmk.Thom-Whitney}
If the ground field is any algebraically closed field of characteristic zero, instead of Dolbeault forms, the associated dgLa is defined via the Thom-Whitney complex associated with the sheaf of endomorphisms $\End E$ (see \cite{FMM}). 
\end{remark}

\section {Deformation of locally free sheaves with a fixed subspace of sections} \label{sect.def(E,U)}

Let $E$ be a locally free sheaf of $O_X$-modules on a smooth projective variety $X$ and fix a subspace $U \subseteq H^0(X,E)$.
In this section, we study infinitesimal deformations of $E$ which preserves the subspace $U$.
We point out that in the literature such a pair $(E,U)$ is called a \emph{local system} of type $(n=\rk E,d=\deg E, k=\dim U)$. Deformations of local systems, stability conditions for them and the concerned moduli space are studied in \cite{lepoit, king,bgmm03, he}. 

\smallskip

We start with some definitions and results of \cite{elenatesi, ElenaFormal}.

\begin{definition} \label{def.(E,U)}
Let $A$ be a local Artinian $\kk$-algebra with residue field $k$. An \emph{infinitesimal deformation} of the pair $(E, U)$  over $A$  is the data $(E_A,\pi_A, i_A)$ of:
\begin{itemize}
\item a deformation $(E_A, \pi_A)$ of $E$ over $A$,
\item a morphism $i_A:U\otimes A \rightarrow H^0(E_A)$,
\end{itemize}
such that the following diagram commutes
\begin{equation}  \label{def(E,U)}
\xymatrix{ U\otimes A \ar[d]^{\pi} \ar[r]^{i_A} & H^0(E_A) \ar[d]^{\pi_A}\\
U \ar@{^{(}->}[r]^{i} & H^0(E).   }
\end{equation}

Two of such deformations  $(E_A,\pi_A,i_A)$, $(E'_A,\pi'_A,i'_A)$ are \emph{isomorphic} if there exist an isomorphism $\phi:E_A \rightarrow E'_{A}$ of sheaves of $\OO_X \otimes A$-modules,  such that
 $\pi'_A \circ \phi=\pi_A$ as in diagram \eqref{def(E)},
and an isomorphism $\psi:U\otimes A \to U\otimes A$,  that makes the diagram commutative:
\begin{equation}  \label{isodef(E,U)}
\xymatrix{ U\otimes A \ar[d]^{\psi}   \ar[r]^{i_A}& H^0(E_A) \ar[d]^{\phi} \\
             U\otimes A \ar[r]^{i'_A}& H^0(E'_A) . }
\end{equation}

Note that, this implies that $\phi$ induces an isomorphism  $\phi: i_A(U\otimes A) \rightarrow i'_A(U\otimes A)$. 

The \emph{functor} of infinitesimal deformations of $(E,U)$ is
\[ \Def_{(E,U)}: \Art_\kk \to \Set,\]
that associates with every $A\in \Art_\kk$ a deformation $(E_A,\pi_A, i_A)$ as defined above. In the following, we will often shorten the notation of such a deformation with $(E_A, i_A)$. 
\end{definition}

\begin{proposition}
The functor $\Def_{(E,U)}: \Art_\kk \to \Set$ defined above is a deformation functor.
\end{proposition}
\begin{proof}
First observe that $\Def_{(E,U)}(\kk)=\{(E,i)\}$, where $i:U \to H^0(E)$ is the inclusion and so $\Def_{(E,U)}$ is a functor of Artin rings.  

To prove it is a deformation functor, we verify the two conditions of Definition \ref{def. def functor}.
\begin{itemize}
\item Let $B \rightarrow A$ and $C \rightarrow A$ two morphisms of Artin rings, suppose the first one to be surjective, we have to prove that
\[ \eta: \Def_{(E,U)}(B \times_A C) \rightarrow \Def_{(E,U)}(B) \times_{\Def_{(E,U)}(A)} \Def_{(E,U)}(C)  \] is surjective.
Let $((E_B,i_B), (E_C,i_C)) \in \Def_{(E,U)}(B) \times_{\Def_{(E,U)}(A)} \Def_{(E,U)}(C)$ and let $(E_A, i_A)$ be the deformation over $A$ to which both reduce. It is classically known (see for example \cite[Prop.3.2]{Schlessinger} for the line bundle case), that $\widetilde{E}:= E_B \times_{E_A} E_C$ is a locally free sheaf of ${\mathcal O}_{B\times_A C}$-modules that deforms $E$ and which reduces to $E_B$ and $E_C$ over $B$ and $C$, respectively. By hypothesis,  $i_B \otimes_B \Id_A = i_C \otimes_C \Id_A = i_A: U \otimes A \to H^0(E_A)$, that means that $U$ is a subspace of sections of $E$ that lift to $E_B$ and $E_C$. Thus, there exists $\tilde{i} :=i_B \times i_C: U \otimes (B\times_A C) \to H^0({\widetilde E})$ and $({\widetilde E}, \tilde{i}) \in \Def_{(E,U)}(B \times_A C)$ proves the surjectivity of $\eta$. 
\item Let now $A=\kk$, we have to prove that $\eta$ is bijective. The surjectivity is done. Suppose now that $(\hat{E}, \hat{i}) \in \Def_{(E,U)}(B\times_\kk C)$ is an other deformation of $(E,U)$ sent to  $((E_B,i_B), (E_C,i_C)) $ under $\eta$. Since $\hat{E}$ and $\widetilde{E}$ both reduce to $E_B$, $E_C$, $E$ over $B$, $C$ and $k$ respectively, it is classically known  (see \cite[Prop.3.2]{Schlessinger}), that they are isomorphic. Note now that $\hat{i}: U \otimes (B\times_\kk C) \to H^0(\hat{E})$ is completely determined by its reductions over $B$ and $C$, that are respectively $\hat{i}\otimes_{B\times_\kk C} B=i_B$ and $\hat{i}\otimes_{B\times_\kk C} C=i_C$. Thus $\hat{i}$ and $\widetilde{i}$ have to coincide. 
\end{itemize}

\end{proof}
\medskip

There is a natural transformation of functors
\[ \Def_{(E,U)} \to \Def_E,\]
that associates with every deformation of the pair $(E, U)$ over $A\in \Art_\kk$, the deformation of the sheaf $E$ over $A$ forgetting the deformed space of sections.  
\begin{lemma} \label{rel obst def functors}
The relative obstruction theory of the natural transformation $\Def_{(E,U)} \to \Def_E$ is contained in $\Hom(U, H^1(X,E))$. 
\end{lemma}
\begin{proof}
Let $0 \to J \to B \to A \to 0$ be a small extension. 
Let  
\[ \left( (E_A, i_A), E_B \right) \in \Def_{(E,U)}(A) \times_{\Def_{E}(A)} \Def_{E}(B), \]
thus $E_A$ is a deformation of $E$ over $A$ that lifts to a deformation $E_B$ over $B$. 
Consider the exact sequence
 \[   0 \to E \otimes J \to E_B \to E_A \to 0, \]
that induces the exact sequence in cohomology
\[ 0 \to H^0(E) \otimes J \to H^0(E_B) \to H^0(E_A) \stackrel{\delta}{\to} H^1(E) \otimes J \to \ldots. \]
Note that a section $s \in H^0(E_A)$ lifts to a section of $E_B$ if and only if its image under the boundary map $\delta$ in $H^1(X,E)\otimes J$ is zero. 
Thus, the obstructions of $\Def_{(E,U)}$ relative to $\Def_E$ are contained in $\Hom(U, H^1(X,E)) \otimes J$ and the obstruction theory is complete.
 \end{proof}

\medskip
From now on, the base field will be $\CC$. 
Consider the complex of sheaves of differential forms on $X$ with values in the sheaf $E$ with the Dolbeault differential 
\[ 0 \to  \mathcal{A}_X^{0,0}(E) \stackrel{\deltabar}{\to}  \mathcal{A}_X^{0,1}(E) \stackrel{\deltabar}{\to}  \mathcal{A}_X^{0,2}(E) \stackrel{\deltabar}{\to} \ldots \]
and the sheaf ${\mathcal H}om( \mathcal{A}_X^{0,*}(E), \mathcal{A}_X^{0,*}(E))$ of homorphisms of this complex. 
Note that the graded vector space of global sections of the sheaf ${\mathcal H}om( \mathcal{A}_X^{0,*}(E), \mathcal{A}_X^{0,*}(E))$ is the same as the graded vector space of homomorphisms of the complex of global sections
\[ 0 \to   A_X^{0,0}(E) \stackrel{\deltabar}{\to}  A_X^{0,1}(E) \stackrel{\deltabar}{\to}  A_X^{0,2}(E) \stackrel{\deltabar}{\to} \ldots \]
We denote it as $\Hom^*(A_X^{0,*}(E),A_X^{0,*}(E))$. 

As always, when one considers the homomorphism of a complex, one can endow $\Hom^*(A_X^{0,*}(E),A_X^{0,*}(E))$  with an obvious structure of dgLa using as bracket the wedge product on forms and the composition of homomorphism and as differential 
the bracket with the differential of the complex.  
The dgLa $\Hom^*(A_X^{0,*}(E),A_X^{0,*}(E))$ controls the deformation of the complex $(A_X^{0,*}(E), \deltabar)$, as proved in \cite[Section 4]{ManettiSemireg}.

Note that there exists an inclusion of dgLas  
\begin{equation}  \label{morfismo da end in hom(h0,h1)}\phi: A_X^{0,*}(\End E)  \to \Hom^*(A_X^{0,*}(E),A_X^{0,*}(E)), \end{equation}
defined for  $\omega \cdot f \in A_X^{0,p}(\End E) $ and $\eta \cdot s \in A_X^{0,q}(E)$ as
\[\phi ( \omega \cdot f)(\eta \cdot s)= \omega \wedge \eta \cdot f(s) \in A_X^{0,p+q}(E).
\]
It is easy to see that the elements in $A_X^{0,*}(\End E)$ correspond to the morphism of the complex $A_X^{0,*}(E)$ that are  $A_X^{0,*}$-linear. Moreover, the Maurer-Cartan elements of $A_X^{0,*}(\End E)$ which are equivalent to zero in $\Hom^*(A_X^{0,*}(E),A_X^{0,*}(E))$ under the inclusion  $\phi$ correspond to the deformations of $E$ that preserve the dimension of the cohomology spaces $H^i(X,E)$ for every index $i$, as proved in  \cite[Lemma 4.1]{ManettiSemireg}.

\medskip

Next, consider  the complex 
 \[Q_U: 0\to U  \stackrel{i}{\hookrightarrow}  A^{0,0}_X(E)  \stackrel{\deltabar}{\to}   A^{0,1}_X(E)  \stackrel{\deltabar}{\to}  A_X^{0,2}(E) \stackrel{\deltabar}{\to} \ldots,  \]
where $U$ is in degree -1.
We define  the graded vector space
\[ D_U = \left\{ f\in \Hom^*(Q_U,Q_U) \mid f|_{A^{0,*}_X(E) } \in  {A}^{0,\ast}_X(\End E)  \right\}.   \]
For any element $f \in D_U^j$, we use the notation as pair $f=(f_{-1}, f_i)$, where $f_{-1}:U \to A^{0,j-1}_X(E)$ and $f_i \in  A^{0,j}_X(\End E) $.
Endowed with the same differential and bracket as  $\Hom^*(Q_U,Q_U)$, $D_U$ is a dgLa. 
In particular,  the tangent space to $\Def_{D_U}$ is $ \Def_{D_U}(\CC[\epsilon])=H^1(D_U)$ and the obstructions to deformations are contained in $H^2(D_U)$.

Consider the morphism:
\[r: D_U\to  A^{0,\ast}_X(\End E), \] 
that associates with any $f=(f_{-1}, f_i)\in D_U$, the element $f_i \in  A^{0,*}_X(\End E)$. By definition, it is a morphism of dgLas and it is clearly surjective.
Denoting by $M^*=\ker r= \{ f \in D_U\ \mid \  f|_{ A^{0,*}_X(E)}=0 \}$,   we have the following  short exact sequence of dgLas
\[
0 \to M^* \to D_U \to A^{0,\ast}_X(\End E) \to 0,
\]
that induces the following exact sequence in cohomolgy
\begin{equation}
\begin{split} \label{successione lunga comologia DU}
0 \to H^0(M^*) \to H^0(D_U) \to H^0(A^{0,\ast}_X(\End E)) \to \\
 \to H^1(M^*) \to H^1(D_U) \to H^1(A^{0,\ast}_X(\End E)) \to \\
  \to H^2(M^*) \to H^2(D_U) \to H^2(A^{0,\ast}_X(\End E)) \to \ldots \\
 \end{split}
\end{equation}

Since $ A^{0,\ast}_X(\End E)$ is the Dolbeault resolutions of the sheaf $\End E$, there are isomorphisms  $H^j(A^{0,\ast}_X(\End E)) \cong H^j(X,  \End E)$, for all $j \geq 0$. 
Note that as dg vector space $M^*$ is isomorphic to $\Hom^*(U, Q_U)$, where $U$ is considered as a dg-vector space concentrated in degree -1, thus  $H^0(M^*)\cong \Hom (U, H^{-1}(Q_U))=0$,   $H^1(M^*)\cong \Hom (U, H^0(X,E)/U )$ and $H^j(M^*)\cong \Hom (U, H^{j-1}(X,E))$, for $j \geq 2 $.

Therefore the long exact sequence \eqref{successione lunga comologia DU} becomes 
 
\begin{equation} 
\begin{split} \label{seconda successione lunga comologia DU}
0 \to H^0(D_U) \to H^0( X,\End E) \to \\
 \to \Hom (U, H^0(X,E)/U  ) \to H^1(D_U) \to H^1(X,\End E)   \stackrel{\alpha}{\to}  \\
  \to \Hom (U, H^{1}(X,E))  \stackrel{\beta}{\to}  H^2(D_U)  \stackrel{\gamma}{\to}  H^2(X,\End E) \to \cdots \\
 \end{split}
\end{equation}
where the map $\alpha$ is  the restriction to $U$ of the morphism induced in cohomology by the inclusion $\phi$ defined in \eqref{morfismo da end in hom(h0,h1)}.

\smallskip

Note, that the dgLa morphism $r: D_U\to A^{0,\ast}_X(\End E)$ induces a natural transformation of functors:
\[ \Def_{D_U} \to \Def_{A^{0,*}_X(\End E)}.\]

\begin{lemma} \label{rel obst dgla functors}
A complete relative obstruction theory of the natural transformation 
$\Def_{D_U} \to \Def_{A^{0,*}_X(\End E)}$ is contained in $\Hom(U, H^1(X,E))$. 
\end{lemma}
\begin{proof}
Let $0 \to J \to B \to A \to 0$ be a small extension. 
Let 
\[
x=\left((x_{-1}, x_i),  \tilde{x}_i)\right)  \in \Def_{D_U}(A) \times_{\Def_{A^{0,*}_X(\End E)}(A)} \Def_{A^{0,*}_X(\End E)}(B),\]
thus $x_i \in \MC_{A^{0,*}_X(\End E)} (A)$ lifts to $\tilde{x}_i \in \MC_{A^{0,*}_X(\End E)}(B)$. 
Choose a lifting  $\tilde{x}_{-1} \in \Hom(U, A^{0,0}_X(E)) \otimes B$ of $x_{-1}$ and define $\tilde{x}=(\tilde{x}_{-1}, \tilde{x}_i)$. 
The relative obstruction of $x$ is the class of $ob(x)=d\tilde{x} + \frac{1}{2}[\tilde{x}, \tilde{x}] \in H^2(D_U) \otimes J$. 
Tensoring the sequence in \eqref{seconda successione lunga comologia DU} with $J$, we get the exact sequence:
\[ \ldots \to \Hom (U, H^{1}(X,E))\otimes J \to H^2(D_U) \otimes J  \stackrel{\gamma}{\to} H^2(X,\End E) \otimes J  \to \cdots \] 
Since the element $ob(x)$ goes to zero under the map $\gamma$, 
 the relative obstruction $ob(x)$ is contained in $\Hom (U, H^{1}(X,E)) \otimes J$.

The defined obstruction is complete. Indeed, if there exists a lifting $\tilde{x}\in \Def_{D_U}(B)$ of $x$, it satisfies the Maurer-Cartan equation, thus $ob (x) =d\tilde{x} + \frac{1}{2}[\tilde{x}, \tilde{x}]=0$.   \end{proof}

The following proposition is one of the main result of this section.

\begin{proposition}\cite[Corollary 4.1.14]{elenatesi} \label{dgla D controlla def (E,U)}
The dgLa $D_U$ controls deformations of the pair $(E,U)$. 
The isomorphism of functors is given, for all $A\in \Art_\CC$, by 
$$\begin{array}{rrll}
\Phi: & \Def_{D_U}(A)& \longrightarrow &\Def_{(E,U)}(A) \\
& x& \longrightarrow & \left( (\ker(\deltabar+x_0), \Id+x_{-1} \right)
\end{array}$$
\end{proposition}
\begin{proof} 
 For completeness and clearness we write here the proof. 
We leave to the reader the classically known calculations  for 
the isomorphism of the functors of  Proposition \ref{dgla che governa Def E}.
We divide the proof in two steps.

\smallskip
\noindent{\bf{First step: the natural transformation of functors $\Phi$ is well defined.}}
 Let $x=(x_{-1},x_i)\in D_U^1 \otimes m_A$  be a  Maurer-Cartan element and prove that it  defines a deformation of the pair $(E,U)$.
It is a classical fact that $E_A:=\ker(\deltabar+x_0)$ with the map $\pi_A := \Id \otimes \pi$ defines a locally free sheaf that is deformation of the sheaf $E$ (Proposition \ref{dgla che governa Def E}). The map 
  $i_A:= \Id+x_{-1}$
fits in the diagram \eqref{def(E,U)}, in particular $i_A(U\otimes A) \subset H^0(X, E_A)$. Indeed, 
\begin{eqnarray*} 
(\deltabar+x_0)\circ(\Id+x_{-1})|_{U\otimes A} &= &\deltabar\circ\Id+ \deltabar\circ x_{-1}+x_0\circ\Id+x_0\circ x_{-1}\\
&=& 0+(\deltabar\circ x_{-1}+x_0\circ\Id)+x_0\circ x_{-1}\\
&=&(dx)_{-1}+\frac{1}{2}[x,x]_{-1}=0,
\end{eqnarray*}
since $U\subset H^0(X,E)$ and $x \in \MC_{D_U}(A)$.
Then, the maps $i_A$ and $\pi_A$ makes the diagram \eqref{def(E,U)} commutative. Indeed, since $x_{-1} \in D_U \otimes m_A$:
\[ \pi_A \circ i_A|_{U\otimes A} = (\Id \otimes \pi) \circ (\Id + x_{-1})|_{U\otimes A} = (\Id \otimes \pi)|_{U\otimes A} +  (\Id \otimes \pi) \circ x_{-1}  
= \pi + 0= i \circ \pi. \]

Moreover, the morphism above is well defined on deformation functors. Let $x,y \in \MC_{D_U}(A)$ be two gauge equivalent elements via $z\in D_U^0 \otimes m_A$, i.e. $e^z *x=y$.
For $i\geq 0$, the elements $e^{z_i}: A^{0,i}_X(E) \otimes A \to A^{0,i}_X(E) \otimes A $ define an isomorphism of degree zero and, as classically know, the gauge relation is equivalent to the commutativity 
\[ \deltabar+y_0= e^{[z_i,-]}(\deltabar+x_0)=e^{z_{i+1}}\circ(\deltabar+x_0)\circ e^{-z_i}. \]
Thus, $\phi:=e^{-z_0}$ defines an isomorphism between the deformed sheaves $\ker(\deltabar + x_0)$ and $\ker(\deltabar + y_0)$. 

Similarly, the element $\psi:=e^{z_{-1}}: U \otimes A \to U\otimes A$ defines an isomorphism and the gauge relation is equivalent to the commutativity of the diagram \eqref{isodef(E,U)}. Indeed, 

\begin{eqnarray} \label{eqgauge2}
 y_{-1} &=&  e^{z}* x_{-1} =   x_{-1}+\sum_{n=0}^{+\infty} \frac{([z,-])^n}{(n+1)!}([z,x]_{-1}-(dz)_{-1})=\nonumber \\ 
&=& x_{-1}+\sum_{n=0}^{+\infty} \frac{([z,-])^n}{(n+1)!}([z,x]_{-1}+[z,\Id]_{-1})=
x_{-1}+\sum_{n=1}^{+\infty} \frac{([z,-])^n}{n!}(\Id+x_{-1})=\nonumber \\
&=& \sum_{n=0}^{+\infty} \frac{([z,-])^n}{n!}(\Id+x_{-1})-\Id= e^{[z,-]}(\Id+x_{-1})-\Id  \ , 
\end{eqnarray}
where we use $(dz)_{-1}=i\circ z_{-1}-z_0\circ i=-[z,\Id]_{-1}$. Thus:
$$ \Id+ y_{-1}= e^{[z,-]}(\Id+x_{-1})=e^{z_{0}}\circ(\Id+x_{-1})\circ e^{-z_{-1}},$$
as we wanted. 

\smallskip
\noindent{\bf{Step two: $\Phi$ is an isomorphism of functors.}}\\
First the injectivity of $\Phi(A)$ for every $A\in \Art_\CC$. 
Suppose that $x=(x_{-1}, x_i)$ and $ y=(y_{-1}, y_i) \in \MC_{D_U}(A)$ induce isomorphic  deformations $(\ker(\deltabar + x_0), \Id+ x_{-1})$ and $ (\ker(\deltabar + y_0), \Id+ y_{-1})$ via the isomorphisms $(\phi, \psi)$, as in Definition \ref{def.(E,U)}.   

It is classical to lift $\phi$ to an isomorphism of the form $e^z$, with $z \in A^{0,0}(\End E) \otimes m_A$ and to get the following commutative diagram  
\begin{equation} \label{iniett}
\xymatrix{
0 \ar[r]& U\otimes A \ar[d]^{\psi}\ar[r]^{\Id+x_{-1}\ \ \ }& \ker(\deltabar+x_0)\ar[r]^{i}\ar[d]^{\phi=e^z} & A_X^{(0,0)}(E)\otimes A \ar[r]^{\ \ \ \ \ \deltabar+x_0}\ar[d]^{e^z} & \cdots \\
0 \ar[r]& U\otimes A \ar[r]^{\Id+y_{-1} \ \ \  } &\ker(\deltabar+y_0)\ar[r]^{i}       & A_X^{(0,0)}(E)\otimes A \ar[r]^{\ \ \ \ \ \deltabar+y_0} & \cdots }
\end{equation}
The isomorphism $\psi$ is of the form $e^w$, with $w \in \Hom(U,U)\otimes m_A$ too, because it is the identity on the residue field. 
Thus there exists an element  $t=(w,z) \in D_U^0\otimes m_A$, such that $e^t$ is an isomorphism that makes the diagram (\ref{iniett}) commutative. It is an easy calculation, similar to (\ref{eqgauge2}), to see that the commutativity is equivalent to the gauge relation $y=e^t* x$.

Moreover, by next Proposition \ref{liscezza morf funtori}, the morphism of functor $\Phi$ is smooth, thus $\Phi(A)$ is surjective, for all $A\in \Art_\CC$. 
\end{proof}

\begin{proposition} \label{liscezza morf funtori}
The morphism of functors 
$\Phi: \Def_{D_U} \longrightarrow \Def_{(E,U)}$ defined in Proposition \ref{dgla D controlla def (E,U)} 
is smooth.
\end{proposition}

\begin{proof}  Let $0\to J \to B\to A \to 0$ be a small extension. Let $x =(x_{-1}, x_i) \in \Def_{D_U}(A)$ and $\Phi(x)=(E_A,i_A) \in \Def_{(E,U)}(A)$. The smoothness of $\Phi$ is equivalent to saying that $x$ lifts to an element $\tilde{x}\in \Def_{D_U}(B)$
if and only if  $(E_A,i_A)$ lifts to a pair $(E_B,i_B)\in \Def_{(E,U)}(B)$. 
One direction is obvious. 

For the other one, we recall that the morphism of functors
$\Psi: \Def_{A^{0,*}(\End E)} \to \Def_E$, defined in Proposition \ref{dgla che governa Def E}, is smooth. Thus it is enough to show that the relative obstruction theories of Lemmas \ref{rel obst def functors} and \ref{rel obst dgla functors} are isomorphic via the correspondence between the Dolbeault and \v{C}ech cohomology. 

As in Lemma \ref{rel obst dgla functors}, let 
\[
x=\left((x_{-1}, x_i),  \tilde{x}_i)\right)  \in \Def_{D_U}(A) \times_{\Def_{A^{0,*}_X(\End E)}(A)} \Def_{A^{0,*}(\End E)}(B)
\]
and let $ob(x)\in \Hom(U, H^1(X, E))\otimes J$ be its obstruction. Observe that here $H^1(X, E)$ is the Dolbeault cohomology group and let us find the element in \v{C}ech cohomology that corresponds to $ob(x)$.
For every $s \in U \otimes A$, $ob(x)(s) \in H^1(X,E)\otimes J$. This class is represented by a closed element in $A_X^{0,*}(E)\otimes J$, denoted again by $ob(x)(s)$, which is then locally exact. Therefore there exist an open cover $\mathcal W=\{ W_i\}$ of $X$ and $\tau_i(s) \in A_{W_i}^{0,0}(E) \otimes J$, such that $\deltabar \tau_i(s)= ob(x)(s)|_{W_i} $.
Define on $W_i \cap W_j$ the elements $\sigma_{ij}(s)=\tau_i(s)- \tau_j(s)$,   they are \v{C}ech cocycles and their class $[\{\sigma_{ij}(s)\}_{ij}] \in H^1(X,E)\otimes J$ defines the corresponding element $ob(x)(s)$ in \v{C}ech cohomology.  

As in Lemma  \ref{rel obst def functors}, let
\[ \left( (E_A, i_A), E_B \right) \in \Def_{(E,U)}(A) \times_{\Def_{E}(A)} \Def_{E}(B), \]
where $\Phi(x)=(E_A,i_A)$ and $E_B=\Psi( \tilde{x}_i)$. For every $s \in U\otimes A$, the obstruction to lift $i_A(s)\in H^0(E_A)$ to a section of $E_B$ lives in $H^1(X,E)\otimes J$ and is given by $\delta(i_A)(s)$, where $\delta$ is the coboundary map
\[\ldots \to H^0(E_B) \to H^0(E_A) \stackrel{\delta}{\to} H^1(E) \otimes J\to \ldots. \]
Recall that the construction of the coboundary map is obtained by chasing the following
diagram
$$\xymatrix{ 0\ar[r] & \check{C}^0(\mathcal{W},E)\otimes J\ar[r]\ar[d]_{\check{\delta}}& \check{C}^0(\mathcal{W},E_B) \ar[r]\ar[d]_{\check{\delta}}& \check{C}^0(\mathcal{W},E_A)\ar[r]\ar[d]_{\check{\delta}}&   0 \\
0\ar[r] & \check{C}^1(\mathcal{W},E)\otimes J\ar[r]& \check{C}^1(\mathcal{W},E_B)  \ar[r]& \check{C}^1(\mathcal{W},E_A) \ar[r]&   0.   } 
$$  
The element $(i_A)(s)=\{ i_A(s)|_{W_i}\}_i \in H^0(E_A)$ can be lifted to an element $i_A(s)|_{W_i} - \tau_i (s) \in \check{C}^0(\mathcal{W}, E_B)$. Applying the \v{C}ech differential to it, we get $\check{\delta} (i_A(s)|_{W_i} - \tau_i (s))= \{i_A(s)|_{W_i}-\tau_i (s)  - i_A(s)|_{W_j}+\tau_j (s)\}_{ij} =\{ \tau_i (s)- \tau_j (s)\}_{ij} = \{\sigma_{ij}(s)\}_{ij}$. 
As we state, the two obstructions coincides. 
\end{proof}

 As a direct consequence of Proposition \ref{dgla D controlla def (E,U)}, we get the following result,  already obtained in \cite[Th\'eoreme 3.12]{he}. See also  \cite[Proposition 3.4]{bgmm03} for the curve case. 

\begin{corollary} \label{cor.sp tang e ostr Def(E,U)}
The tangent space to $\Def_{(E,U)}$ is $ H^1(D_U)$  and all obstructions are contained in $H^2(D_U)$.  
\end{corollary}

\begin{remark}   
If the ground field is any algebraically closed field of characteristic zero,  in the same spirit as for deformations of the sheaf $E$ (see  Remark  \ref{rmk.Thom-Whitney}), we expect to define a dgLa that controls deformations of $(E,U)$ using the Thom-Whitney complex associated with the sheaf of homomorphism of a suitable complex of sheaves \cite{IM}. \end{remark} 

\smallskip

In the following, we briefly focus on smoothness of the forgetful morphism $r: \Def_{(E,U)} \to \Def_E$. 
The following corollary is a direct consequence of Lemmata \ref{rel obst def functors} and \ref{rel obst dgla functors}. Otherwise, it can be obtained applying Theorem \ref{teo. liscio se tangete suriettivo e iniettivo ostruzione} to the exact sequence \eqref{seconda successione lunga comologia DU}. 

\begin{corollary} \label{prop.rel obstr classica}
If  $\Hom(U, H^1(E))=0$, the forgetful morphism of functors $r:\Def_{(E,U)} \to \Def_E$ is  smooth. 
\end{corollary}

\begin{remark}  \label{cor. alpha sur} 
By Proposition  \ref{prop liscezza funtori}, the smoothness of the forgetful morphism $r: \Def_{(E,U)} \to \Def_E$ implies the equivalence between the smoothness of the two functors  $\Def_E$ and $\Def_{(E,U)}$.
\end{remark}
 
\begin{corollary} \label{corollario alpha su implica r liscio}
If the map $\alpha: H^1(X,\End E) 
\to \Hom (U, H^{1}(X,E))$ that appears in \eqref{seconda successione lunga comologia DU} is surjective, then  the forgetful morphism $r: \Def_{(E,U)} \to \Def_E$ is smooth. 
\end{corollary}

\begin{proof}
Let $0\to J \to B \to A \to 0$ be a small extension in $\Art_\CC$ and consider
\[
x=\left((x_{-1}, x_i),  \tilde{x}_i)\right)  \in \Def_{D_U}(A) \times_{\Def_{A^{0,*}_X(\End E)}(A)} \Def_{A^{0,*}_X(\End E)}(B).\]
Since $x_i$ lifts to $\tilde{x}_i$, from the diagram of obstruction theories 
\begin{equation} \label{ostr}
\xymatrix{ \Def_{D_U}(A) \ar[r]^-{ob} \ar[d] & H^2(D_U) \otimes J \ar[d]^{\gamma} \\
           \Def_{A^{0,*}_X(\End E)} (A)\ar[r]^-{ob}          & H^2(\End E) \otimes J                            }
\end{equation}
 we get that the relative obstructions to lift $x$ to an element in $\Def_{(E,U)}(B)$ is contained in $\ker \gamma$. This kernel is trivial: Indeed, looking at \eqref{seconda successione lunga comologia DU}, 
the surjectivity of the map $\alpha: H^1(X,\End E)  \to \Hom (U, H^{1}(X,E))$ implies that the morphism $\gamma : H^2(D_U) 
\to H^2(\End E) $  is injective. 
\end{proof}

\begin{remark}\label{remark alpha su allora Hom(U, H^1(E))=0}
The condition $\Hom(U, H^1(E))=0$ is equivalent to the surjectivity of the map $\alpha: H^1(X,\End E)  \to \Hom (U, H^{1}(X,E))$.
Indeed,  by Corollary \ref{corollario alpha su implica r liscio}, if $\alpha$ is surjective, then $r$ is smooth and also the map  $H^1(D_U) \to H^1(X,\End E)$ on the tangent spaces of the functors is surjective. By the exact sequence \eqref{seconda successione lunga comologia DU},  the map $\alpha$ is actually the zero map and so  $ \Hom (U, H^{1}(X,E)=0$.

The other implication is obvious.
\end{remark}

\begin{corollary}   In the notation above, we have 
\[ \dim t_{\Def_{(E,U)}} \geq \dim t_{\Def_{E}}
 - k \cdot \dim H^1(X,E),\]
 where $k$ is the dimension of $U \subseteq H^0(X,E)$.  
\end{corollary}

\begin{proof}
By the long exact sequence \eqref{seconda successione lunga comologia DU}, 
\[
 \cdots \to \Hom (U, H^0(X,E)/U  ) \to H^1(D_U)  \stackrel{\beta}{\to} H^1(X,\End E)   \stackrel{\alpha}{\to}   \Hom (U, H^{1}(X,E))  \to   \cdots \\
\]
we have
\[
\dim t_{\Def_{(E,U)}} = \dim H^1(D_U) \geq \dim \im \ \beta = \dim \ker  \ \alpha=  \dim H^1(X,\End E) - \dim \im \alpha
\]
\[
  \geq   \dim H^1(X,\End E) - \dim    \Hom (U, H^{1}(X,E))= \dim t_{\Def_{E}}  -k \cdot \dim H^1(X,E).
\]
\end{proof}

Using our description of deformations via dgLas, we can generalise a classical result. 
Fix a section  $s \in H^0(X,E)$, the morphism $\phi$ of  \eqref{morfismo da end in hom(h0,h1)} induces in cohomology the cup product
\[
- \cup s:  H^1(X, \End E) \to H^1(X,E),
\]
where $a\cup s=\alpha(a)(s)$, for every  $a \in H^1(X, \End E)$.

\begin{proposition} \label{prop.sezione si estende se cup product va a zero}
Let $E$ be a locally free sheaf over a projective variety $X$. A section $s  \in H^0(X,E)$ can be extended to a section of a first order deformation of $E$ associated to an element  $a \in H^1(X, \End E)$ if and only if $a \cup s =0 \in H^1(X,E)$.
\end{proposition}
\begin{proof}
Let $s \in H^0(X,E)$ be a section and define $U=\langle s \rangle$.
Recalling our descriptions via dgLas of the first order deformations given after Proposition \ref{dgla che governa Def E} and in Corollary \ref{cor.sp tang e ostr Def(E,U)}, 
we can rewrite the exact sequence (\ref{seconda successione lunga comologia DU}) as
\[ \ldots \to H^1(D_U)= \Def_{(E,U)}(\CC[\epsilon])\stackrel{r}{\to} H^1(X, \End E)=  \Def_E(\CC[\epsilon]) \stackrel{\alpha}{\to} \Hom(U, H^1(X,E)) \to \ldots \] 
The section $s$ can be extended to a deformation associated to $a \in H^1(X, \End E)=  \Def_E(\CC[\epsilon])$ if and only if $a \in \im r$. Since $\im r= \ker \alpha$ we have the required description. 
\end{proof}
The same result is classically known for line bundles over a curve (see \cite[Lemma page 186]{ACGH}) and for line bundles over a projective variety (see \cite[Proposition 3.3.4]{Sernesi}).
 
\smallskip

This result can be reinterpreted in terms of some special maps and it can be seen as a generalization of \cite[Proposition 4.2 (i)]{ACGH}, \cite[Section 2]{Gt} and \cite[Section 4.3]{pardini}.  In the spirit of \cite[Section 4.3]{pardini}, we define a generalization of the Petri map - we will properly introduce in the next section - as the map induced by the cup product:
\[ \mu_0: H^0(X,E) \otimes H^0(X, K_X \otimes E^*) \to H^0(X, K_C \otimes E \otimes E^*), \]
where $K_X$ is the canonical bundle of X, $E^*$ is the dual bundle of $E$ 
and the map
\[ \alpha_n: H^1(X, \End E) \otimes H^{n-1}(\OO_X) \to H^n(\End E), \] 
is given by the cup product. 
Proposition \ref{prop.sezione si estende se cup product va a zero} can be stated saying that for all $\sigma \in H^1(X, \End E) \otimes H^{n-1}(\OO_X)$ and for all $\psi \in H^0(X,E) \otimes H^0(X, K_X \otimes E^*)$ the following cup product vanishes:
\[ \alpha_n(\sigma) \cup \mu_0(\psi) =0,  \]
or equivalently that $\alpha_n\left( H^1(X, \End E) \otimes H^{n-1}(\OO_X)\right)\subset H^n(\OO_X)$ is orthogonal to $\im \mu_0 \subset H^0(X, K_C \otimes E \otimes E^*).$

\bigskip

In the particular case of deformations of pairs $(E, H^0(E))$, the exact sequence
\eqref{seconda successione lunga comologia DU} splits
\[ 0 \to H^1(D_U) \to H^1(X,\End E)   \stackrel{\alpha}{\to}  \Hom (H^0(X,E), H^{1}(X,E)) \to \ldots \]
Thus the tangent space $t_{\Def_{(E,H^0(E))}}= H^1(D_U)$ can be identified with the kernel of the morphism $\alpha:  H^1(X,\End E)  \to   \Hom (H^0(E), H^{1}(X,E))$.

\begin{corollary} \label{Cor. TDef(E,H0(E))}
The tangent space to the deformations of the pair $(E, H^0(E))$ can be identified with
\[ t_{\Def_{(E,H^0(E))}} = \{ a \in H^1(X,\End E) \ \mid \ a \cup s =0, \ \forall  \ s \in H^0(X,E) \}.  \]
\end{corollary}

In the case of line bundles, the description of the tangent space above is also given in \cite[Proposition 3.3, (i)]{pardini}.

\section{Deformations of the pair $(E,U)$ over a curve} 
In this section, we restrict our attention on curves, i.e., we fix a smooth projective curve $C$ of genus $g$ and we study deformations of the pair $(E,U)$, where $E$ is a locally free sheaf of rank $n$ and degree $d$ on $C$ and $U \subseteq H^0(C,E)$ is a subspace of sections of dimension $k$.

\smallskip

First suppose $E=L$ to be a line bundle on $C$. The Petri map, introduced first by Petri in \cite{Petri} and studied deeply in \cite{arbarellocornalba} and \cite{ACGH}, is classically defined as the map induced by the cup product: 
\[\mu_0: U \otimes H^0(C, K_C\otimes L^*) \to H^0(C, K_C), \]
 where $K_C$ denotes the canonical sheaf of $C$ and $L^*$ the dual line bundle of $L$. 
In \cite{bgmm03, Gt, pardini} a generalization of $\mu_0$ to the case of a vector bundle $E$ is introduced as the map induced by the cup product: 
\[ \mu_0: U \otimes H^0(C, K_C \otimes E^*) \to H^0(C, K_C \otimes E \otimes E^*), \]
where $E^*$ is the dual of the vector bundle $E$. 

 Classically for line bundles and also in the successive generalizations [loc. cit.], the Petri map plays a role in the study of the smoothness of the deformations of the pair $(E,U)$ over a curve $C$.  We aim to recover and generalise these kind of results.

Consider the sequence \eqref{seconda successione lunga comologia DU}; in the case of curves, it  reduces to
\begin{equation} 
\begin{split} \label{curva  succesione lunga comologia DU}
0  \to H^0(D_U) \to H^0( C,\End E) 
 \to \Hom (U, H^0(C,E)/U ) \to\\
 \to  H^1(D_U) \to H^1(C,\End E) 
  \stackrel{\alpha}{\to} \Hom (U, H^{1}(C,E)) \to H^2(D_U) \to 0.
 \end{split}
\end{equation}
So we are able to recover  \cite[Proposition 3.4 (i)]{bgmm03}.

\begin{lemma} \label{lemma.equiv h2 zero e petri iniettiva}
In the above notations, the following conditions are equivalent:
\begin{itemize}
\item $H^2(D_U)=0$, 
\item the map $\alpha$ is surjective,
\item $\Hom(U, H^1(E))=0$,
\item the Petri map $\mu_0$ is injective.
\end{itemize}
 \end{lemma}
 \begin{proof}
 The equivalence between the first two conditions follows from the exact sequence (\ref{curva  succesione lunga comologia DU}). The equivalence between the second and the third is Remark \ref{remark alpha su allora Hom(U, H^1(E))=0}. Finally, the second and the last condition are equivalent because $\alpha$ and $\mu_0$ are dual map. Indeed, by Serre duality $H^1(C,\End E)^* \cong H^0(C, K_C \otimes E \otimes E^* )$ and 
 $ (\Hom (U, H^{1}(C,E)) )^*\cong (U^* \otimes H^{1}(C,E))^*\cong U \otimes H^{1}(C,E)^*\cong U \otimes H^0(C, K_C \otimes E^*)$.
\end{proof}

 Aiming to link these conditions with the smoothness of the functor of deformations of $(E,U)$, we prove the following result.
\begin{lemma} \label{lemma.calcolo beta}
In the above notations
\[ h^1(D_U)=  h^2(D_U)+h^0(D_U) + k \chi(E) -\chi(\End E) -k^2,\]
where $\chi(E)$ and $\chi(\End E)$ denote the Euler characteristics of $E$ and $\End E$ respectively.
 \end{lemma}
\begin{proof}
From the above exact sequence \eqref{curva  succesione lunga comologia DU}, we obtain that
\[
h^0(D_U) - h^0(\End E)+ k \cdot \left(h^0(E) -k\right) - h^1(D_U) + h^1(\End E) - k \cdot h^1(E) + h^2(D_U)=0;\]
therefore
\[
\begin{split}h^1(D_U)& =  h^2(D_U) + h^0(D_U) +k \cdot\left(h^0(E) -h^1(E)\right) +h^1(\End E) -h^0(\End E)   -k^2\\
 & =   h^2(D_U)+h^0(D_U) +k \cdot \chi(E) -\chi(\End E) -k^2.
 \end{split}
\]
\end{proof}  

\begin{remark} \label{rmk.calcolo beta}
Let $E$ be a vector bundle of rank $n$ and degree $d$ on a curve $C$ of genus $g$, then
$\chi(E) = d+n(1-g) $ (see \cite{HarrisMorrison} page 154), then $\chi(\End E) = n^2(1-g) $.
Therefore 
\[
\begin{split} k \chi(E) -\chi(\End E) - k^2&= k\left(d+n(1-g)\right) -n^2(1-g) -k^2 \\
&= k (d+n(1-g)) +n^2(g-1) - k^2
\end{split}
\]
\end{remark}

Then, as in  \cite[Definition 2.7]{bgmm03} and \cite[Definition 2.1]{Gt}, we can  introduce the Brill-Noether number.
\begin{definition}
 Let  $E$ be a vector bundle of rank $n$ and degree $d$ on a curve $C$ of genus $g$ and let $U$ be a subspace of sections of dimension $k$.
The \emph{Brill-Noether number} is
\[\beta(n,d,k)= n^2(g-1)  -k (k- d + n(g-1)) +1.\]
\end{definition}
\begin{remark} 
This number is a generalization to vector bundles of the well known \emph{Brill-Noether} number $\rho$ for the data of a degree $d$ line bundle over a curve of genus $g$ with a subspace of sections of dimension $k$:
\[  \rho= g-k(g-d+k),  \]
defined in \cite{arbarellocornalba} and \cite{ACGH}.
As for the classical case of $\rho$, $\beta$ gives an estimate of the dimension of the Brill-Noether loci in the corresponding moduli spaces. 
 \end{remark}

We are now ready to prove our main result of this section. It generalises \cite[Proposition 4.1]{ACGH}, that for a line bundle $L$ on a curve connects the injectivity of the Petri map with the smoothness of the deformations of the pair $(L,U)$ and calculates the dimension of the concerned moduli space in the smooth case. 
\begin{proposition} \label{prop.beta e liscezza}
Let $E$ be a vector bundle of rank $n$ and degree $d$ on the curve $C$ of genus $g$ and let $U$ be a subspace of sections of dimension $k$.
Then, the tangent space to deformations of the pair $(E,U)$ has dimension
\[ \beta(n,d,k)-1+h^0(D_U)  +h^2(D_U). \]
Moreover, the functor $\Def_{(E,U)}$ is smooth and its tangent space has dimension $\beta(n,d,k)-1+h^0(D_U)$ if and only if $H^2(D_U)=0$ if and only if the Petri map is injective.
\end{proposition} 
\begin{proof}
The tangent space to the deformations of the pair $(E,U)$ is $H^1(D_U)$. Then, according to  Lemma \ref{lemma.calcolo beta} and Remark \ref{rmk.calcolo beta},  the dimension of it is given by
 \begin{eqnarray}
h^1(D_U) &=& h^2(D_U)+h^0(D_U) +k \chi(E) -\chi(\End E) -k^2 =\nonumber \\ 
&=&   h^2(D_U) + h^0(D_U) +  k(d+n(1-g)) +n^2(g-1) -k^2 \nonumber \\ 
&= & h^2(D_U) + h^0(D_U)  -1 + \beta(n,d,k), \nonumber
\end{eqnarray} 
As already pointed out in Lemma \ref{lemma.equiv h2 zero e petri iniettiva}, the Petri map is injective if and only if its dual map $\alpha$ is surjective, that is equivalent to the condition that $H^2(D_U)=0$. 
Since the obstructions to deform the pair $(E,U)$ are contained in $H^2(D_U)$, if it vanishes, then the functor $\Def_{(E,U)}$ is smooth and the dimension of the tangent space is easily calculated by the above formula. 
For the other direction, the condition on the dimension of the tangent space implies that $H^2(D_U)=0$. 
\end{proof}

\begin{remark} 
Proposition \ref{prop.beta e liscezza}  is the analogous to \cite[Proposition 3.10]{bgmm03}.
In this article, the author focus their attention on the moduli space of coherent systems from the global point of view. In order to construct a moduli space, they need a suitable notion of stability. The stability conditions, they define, imply that $h^0(D_U)=1$. Thus  
the formula for the  dimension of the tangent space reduces to  $\beta(n,d,k) +h^2(D_U) $ (see Lemma 3.5 [loc. cit]). 
\end{remark}

\section{Deformations of a locally free sheaves and some of its sections}

In this section, we consider a locally free sheaf $E$ of $\OO_X$-modules on a smooth projective variety $X$, such that $\dim H^0(X, E) \geq k$ and study infinitesimal deformations of $E$ such that at least $k$ independent sections of $E$ lift to the deformed locally free sheaf. 

These deformations correspond to  the infinitesimal deformations of the locally free sheaf $E$ induced by an infinitesimal 
deformation of a pair $(E,U)$, for some subspace $U\subseteq H^0(X, E)$ with $\dim U = k$.
In other words, they are the deformations in the image of the forgetful maps of functors:
\[ r_U: \Def_{(E,U)} {\to} \Def_E, \]
for some $U\subseteq H^0(X,E)$, with $\dim U = k$. We denote this subfunctor of $\Def_E$ with $\Def^{k}_E$. 
More explicitly, we give the following definition.
\begin{definition}
Let $E$ be a locally free sheaf of $\OO_X$-modules on a smooth projective variety $X$, such that $h^0(X,E) \geq  k$.
Let $Gr(k,H^0(E))$ be the grassmannian of all subspaces of $H^0(X, E)$ of dimension $k$. We define the functor $\Def^{k}_E:\Art_\kk \to \Set$, that associates with every $A \in \Art_\kk$  the set
\[ \Def^{k}_E (A)= \bigcup_{U \in Gr(k,H^0(E))} r_U(\Def_{(E,U)}(A)).
\] 
 and call it the \emph{functor of deformations of $E$ with at least $k$ sections}.
\end{definition}

\begin{remark}  \label{rmk.no def funct}
In the case $h^0(X,E)=k$, all sections are required to lift to the deformed locally free sheaf and the functor  $\Def^k_E$  is in one-to-one correspondence via the forgetful morphism  with the functor $\Def_{(E,H^0(E))}$, analysed at the end of Section \ref{sect.def(E,U)}.
Thus, our study of $\Def_{(E,H^0(E))}$ applies completely to it and in particular $\Def_E^k$ is in this case a deformation functor.

In general, the functor $\Def_E^k$ is a functor of Artin rings, but unfortunately, it is not a deformation functor.  Indeed, by definition, if $\kk$ is the ground field, we have
\[\Def_E^k(\kk)=  \bigcup_{U \in Gr(k,H^0(E))} r_U(\Def_{(E,U)}(\kk) )= \{E\}, \]
since each of the functors $\Def_{(E,U)}$ are of Artin rings.

Consider now two morphisms of Artin rings $B \rightarrow A$ and $C \rightarrow A$ and suppose one of them to be surjective. 
The map
\[ \eta: \Def_E^k(B\times_A C) \to \Def_E^k(B) \times_{\Def_E^k(A)} \Def_E^k(C) \]
will be in general not surjective. Indeed, let $(E_B, E_C) \in \Def_E^k(B) \times_{\Def_E^k(A)} \Def_E^k(C)$ and let $U$ and $V$ subspaces of sections of $E$ that lift to $E_B$ and to $E_C$ respectively, such that $U \cap V$ has maximal dimension and suppose $\dim U\cap V <k$. Then the existence of a lift of $(E_B, E_C)$ in $ \Def_E^k(B\times_A C)$ will contradict the   maximality of $\dim U\cap V$.
\end{remark}

From now on, we restrict ourself to the field of complex numbers $\CC$. Even if the description of the locus $\Def_E^k(A)$ for $A \in \Art_\CC$ is still quite mysterious, we can explicitly determine the first order deformations and the vector space they generate.  

\begin{theorem}\label{teorema tangente Def^r}
In the above notations, if $h^0(X, E)=k$, the tangent space to the deformation functor $\Def_E^k$ is 
\[  t_{\Def^k_E} = \Def_E^k(\CC[\epsilon]) = \{ a \in H^1(X,\End E) \ \mid \ a \cup s =0, \ \forall  \ s \in H^0(X,E) \}. \]  
If, instead $h^0(X, E) \geq k+1$, 
the first order deformations of $E$ with at least $k$ sections are described by the cone
\[ \Def_{E}^k(\CC[\epsilon]) = \{\nu \in H^1(X, \End E) \mid \exists U \in {Gr(k,H^0(E))}  \mbox{ such that } \nu \cup s =0, \forall s \in U\} \]
and the vector space generated by it, that we call the tangent space to $\Def_E^k$, is 
\[ t_{\Def_{E}^k} = H^1(X, \End E). \]
\end{theorem}

\begin{proof}
As already noticed, in the case $h^0(X, E)=k$, the functor  $\Def^k_E$  is in one-to-one correspondence with the functor $\Def_{(E,H^0(E))}$ and the tangent space is described in Corollary \ref{Cor. TDef(E,H0(E))} to be
\[t_{\Def^k_E}\cong t_{\Def_{(E,H^0(E))}} = \{ a \in H^1(X,\End E) \ \mid \ a \cup s =0, \ \forall  \ s \in H^0(X,E) \}. \]  

If $h^0(X, E) \geq k+1$, by definition, 
\[\Def_E^k(\CC[\epsilon]) = \bigcup_{U \in Gr(k,H^0(E))} r_U(\Def_{(E,U)}(\CC[\epsilon])).\]
For each $U \in  Gr(k,H^0(E))$, we calculate the image of the tangent space to deformations of the pair $(E,U)$ using the exact sequence
\eqref{seconda successione lunga comologia DU}:
\[ \ldots \to H^1(X,D_U) \stackrel{r_U}{\to} H^1(X,\End E)   \stackrel{\alpha_U}{\to}  \Hom (U, H^{1}(X,E)) \ldots \]
Thus
\[  r_U(\Def_{(E,U)}(\CC[\epsilon]))= \ker \alpha_U = \{ \nu \in H^1(X, \End E) \mid \nu \cup s =0, \forall s \in U\}  \]
and the first statement is proved. 

For the second statement, we have to prove that the vector space generated by $\Def_E^k(\CC[\epsilon])$ is the whole space $H^1(X,\End E)$.
One inclusion is obvious. For the other one, it is enough to prove that, for all $s \in H^0(X,E)$ non zero section and for all $w \in H^1(X, E)$, there exists an element $\nu \in \Def_E^k(\CC[\epsilon]) $ such that $\nu (s)=w$.
Since $\dim H^0(X, E) \geq k+1$, it is always possible to find a subspace $U \in  Gr(k,H^0(E))$, such that $s \notin U$ and to build the matrix of $\nu$.
\end{proof}

\begin{remark}
This theorem generalises the classical results for line bundles on curves \cite[Proposition 4.2]{ACGH} and line bundles on a smooth projective varieties \cite[Proposition 3.3]{pardini}.
Our explicit description of the tangent space is also a particular case of the one of the Zariski tangent space to the cohomology jump functors done in \cite[Theorem 1.7]{Budur-Rubio} using dgl pairs.

Moreover, our result is coherent with the well known fact that the locally free sheaves $E$ such that $h^0(X, E) \geq k+1$ are contained in the singular locus of the moduli space of locally free sheaves with al least $k$ independent sections. That is classically obtained defining that moduli space as a determinantal variety (see \cite[Proposition 4.2]{ACGH}, \cite[Theorem 2.8]{bgmm03}, \cite[Corollary 2.8]{costa}, et. al.)  
\end{remark}

In the setting of deformation functors, the next step  after the description of the tangent space is the study of an obstruction space. As well known, in the deformation functors case both spaces have a meaning in term of the corresponding moduli space. Unfortunately, our functor $\Def_E^k$ is not a deformation functor (see Remark \ref{rmk.no def funct}). However, Definition \ref{def.smoothness} holds for $\Def_E^k$  and in the following we try to get some geometrical information linked to its smoothness. 
 
\begin{proposition} \label{prop.alpha sur equivalenza di liscezze}
As above, let $E$ be a locally free sheaf of $\OO_X$-modules on the projective variety $X$, such that $h^0(X,E) \geq k$.
If there exists an $U \in  Gr(k,H^0(X, E))$ such that $\Hom(U, H^1(X, E))=0$ or, in an equivalent way, such that the map $\alpha_U:H^1(X, \End E) \to \Hom(U, H^1(X,E)) $ that appears in \eqref{seconda successione lunga comologia DU} is surjective, then
\[ \Def_E \mbox{ is smooth } \Leftrightarrow \Def_{(E,U)} \mbox{ is smooth } \Leftrightarrow \Def^k_E \mbox{ is smooth. } \]
\end{proposition}
\begin{proof}
From Corollaries \ref{prop.rel obstr classica} and \ref{corollario alpha su implica r liscio}, the two equivalent hypothesis imply that the forgetful morphism $r_U$ is smooth. Then, the first equivalence is a direct consequence of Remark \ref{cor. alpha sur}.  In regard to the second equivalence, since the obstruction is complete, each $E_A \in \Def_{E}^k(A)$ comes from a pair $(E_A, i_A) \in \Def_{(E,U)}(A)$, for every $A\in \Art_\kk$. The above argument implies obviously the equivalence between the smoothness of $\Def_{(E,U)}$ and $\Def_E^k$ . 

\end{proof}

\begin{proposition}  \label{prop.H2=0 liscezza}
 In the above notation, if there exists an $U \in  Gr(k,H^0(E))$ such that $H^2(D_U)=0$, then both the functors $\Def_{(E,U)}$ and $\Def_E^k$ are smooth. 
\end{proposition}
\begin{proof}
Since $H^2(D_U)=0$, the functor $\Def_{(E,U)}$ is smooth and  relative obstruction to $r_U$ is zero, thus $r_U$ is smooth too.
These two properties assure that $\Def_E^k$ is smooth too. 
\end{proof}

\begin{remark}
In general, the hypothesis $H^2(D_U)=0$ implies strictly that $\alpha_U$ is surjective. Since for a curve they are both equivalent to the injectivity of the Petri map (see Lemma \ref{lemma.equiv h2 zero e petri iniettiva}), Proposition \ref{prop.H2=0 liscezza} assures that on a curve $C$, if there exists $U \in  Gr(k,H^0(E))$, such that the Petri map $\mu_0: U \otimes H^0(C, K_C \otimes E^*) \to H^0(C,  K_C \otimes E^*\otimes E)$ is injective, then both the functors $\Def_{(E,U)}$ and $\Def_E^k$ are smooth. See \cite[Proposition 2.1]{casalaina texidor} for a similar result, there the authors assume the injectivity of the Petri map for every $U\in  Gr(k,H^0(E))$. 
\end{remark}

\end{document}